\documentclass[reqno]{amsart}

\usepackage[normalem]{ulem}
\usepackage{xcolor}

\usepackage{geometry}
\usepackage[english]{babel}
\usepackage{latexsym,amsfonts}
\usepackage{amssymb,mathrsfs}
\usepackage{amsaddr}

\usepackage{verbatim}

\def\tr{\mathrm{tr}}

\def\C{\mathscr{C}}

\def\diag{\mathrm{diag}}
\def\GL{\mathrm{GL}}    
\def\Mat{\mathrm{Mat}}   
\def\I{\mathrm{I}}
\def\id{\mathsf{id}}
\def\im{\mathrm{Im}}
\def\PSU{\mathrm{PSU}}
\def\PSL{\mathrm{PSL}}
\def\PSp{\mathrm{PSp}}
\def\Or{\mathrm{O}}
\def\SL{\mathrm{SL}}

\def\SO{\mathrm{SO}}

\def\Sym{\mathrm{Sym}}
\def\Alt{\mathrm{Alt}}

\def\P{{\rm P}}
\def\N{\mathcal{N}}
\def\T{\mathsf{T}}
\def\operp{\perp}

\def\equad{\quad \textrm{ and } \quad}
\newcommand{\Z}{\mathbb{Z}}   
    
\newcommand{\F}{\mathbb{F}}

\newtheorem{theorem}{Theorem}[section] 
\newtheorem{lemma}[theorem]{Lemma}     
\newtheorem{corollary}[theorem]{Corollary}
\newtheorem{proposition}[theorem]{Proposition}

\theoremstyle{definition}

\numberwithin{equation}{section}

\begin{document}

\title[The $(2,3)$-generation of the finite simple odd-dimensional\ldots]{The $(2,3)$-generation of the finite simple\\ 
odd-dimensional orthogonal groups}

\author{M.A. Pellegrini and M.C. Tamburini Bellani}
\email{marcoantonio.pellegrini@unicatt.it}
\email{mariaclara.tamburini@gmail.com}

\address{Dipartimento di Matematica e Fisica, Universit\`a Cattolica del Sacro Cuore,\\
Via della Garzetta 48, 25133 Brescia, Italy}

\begin{abstract}
The complete classification of the finite simple groups that are $(2,3)$-generated is a problem which is still 
open only for 
orthogonal groups. Here, we construct $(2, 3)$-generators for the finite odd-dimensional orthogonal
groups $\Omega_{2k+1}(q)$, $k\geq 4$.
As a byproduct we also obtain $(2,3)$-generators for  $\Omega_{4k}^+(q)$ with $k\geq 3$ and $q$ odd, 
and for $\Omega_{4k+2}^\pm(q)$ with $k\geq 4$ and $q\equiv \pm 1 \pmod 4$.
\end{abstract}

\keywords{Orthogonal group; simple group; generation}
\subjclass[2010]{20G40, 20F05}

\maketitle

\section{Introduction}

A group is said to be $(2,3)$-generated if it can be generated by an involution and an element of order
$3$, equivalently if it is an epimorphic image of $C_2 \ast C_3\cong \PSL_2(\Z)$.
In 1996 (see \cite{LS}) it was shown that the symplectic groups $\PSp_4(q)$, with $q=2^f,3^f$,
are not $(2,3)$-generated and that, apart from the members of these two infinite families and a finite
number of undetermined exceptions, the finite simple classical groups, defined over the Galois field $\F_q$, are $(2,3)$-generated.
Since then, many authors contributed to a constructive solution of the $(2,3)$-generation problem of these groups 
(e.g., see \cite{TW,TWG}).
As a consequence, the list $\mathcal{L}$ of the known exceptions consists now of the following ten groups: 
$\PSL_2(9)$, $\PSL_3(4)$, $\PSL_4(2)$, $\PSU_3(3^2)$, $\PSU_3(5^2)$, $\PSU_4(2^2)\cong \PSp_4(3)$, 
$\PSU_4(3^2)$, $\PSU_5(2^2)$, $\P\Omega_8^+(2)$ and $\P\Omega_8^+(3)$.
This list is complete for linear, unitary and symplectic groups, as shown in \cite{SL12,Unit,Sp}.

In \cite{PT8} we proved that the finite simple $8$-dimensional orthogonal groups are $(2,3)$-gener\-ated,
with the exceptions of $\P\Omega_8^+(2)$ and $\P\Omega_8^+(3)$ found by M. Vsemirnov \cite{Max}.
In this paper, we consider orthogonal groups of dimension $n\geq 9$ and prove the following constructive result.

\begin{theorem}\label{main}
Assume $q$ is odd. The following orthogonal groups are $(2,3)$-generated:
\begin{itemize}
\item[(i)] $\Omega_{2k+1}(q)$ with $k\geq 4$;
\item[(ii)] $\Omega_{4k}^+(q)$ with $k\geq 3$;
\item[(iii)] $\Omega_{4k+2}^+(q)$ with $k\geq 4$ and $q\equiv 1 \pmod 4$;
\item[(iv)] $\Omega_{4k+2}^-(q)$ with $k\geq 4$ and $q\equiv 3 \pmod 4$.
\end{itemize}
\end{theorem}

We recall that the $(2,3)$-generation of $\Omega_5(q)\cong \PSp_4(q)$, when $\gcd(q,6)=1$,  was proved in \cite{CDM} (see also
\cite{PTV}).
Notice that the groups $\Omega_5(3^f)$ are not $(2,3)$-generated, but they are $(2,5)$-generated, see \cite{Ki}.
In \cite{Sp6} it was proved that the groups $\Omega_7(q)$ are $(2,3)$-generated for all odd $q$.
As a consequence of all this, the constructive $(2,3)$-generation  problem for 
the finite simple classical groups  remains open only for the following orthogonal groups:
\begin{itemize}
\item[(i)]  $\P\Omega_{2k}^\pm(q)$ with $k\geq 5$ and $q$ even; 
\item[(ii)] $\P\Omega_{10}^\pm(q)$, $\P\Omega_{14}^\pm(q)$, $q$ odd;
\item[(iii)] $\P\Omega_{4k}^-(q)$ with $k\geq 3$ and $q$ odd; 
\item[(iv)] $\P\Omega_{4k+2}^+(q)$ with $k\geq 4$ and $q\equiv 3 \pmod 4$;
\item[(v)] $\P\Omega_{4k+2}^-(q)$ with $k\geq 4$ and $q\equiv 1 \pmod 4$.
\end{itemize}

In our proof of Theorem \ref{main}, the cases $n\in \{9,11,13,17\}$ are dealt with in Section \ref{11-17}, 
where we use slightly different generators to make the proofs more efficient.
For the general case, the generators are given in Section \ref{gen}. The corresponding proofs are in Section  \ref{gene} 
for   $n\in\{15, 18, 19\}$ or $n\geq 21$, and in Section~\ref{16-20} for $n\in \{12, 16, 20\}$.

\section{Preliminary results}\label{prel}

Let $\F_q$ be the Galois field of order $q = p^f$, a power of the prime $p>2$, and
let $\F$ be the algebraic closure of the field $\F_p$. 
We make $\GL_n(\F)$ act on the left on $V=\F^n$, whose  canonical basis  is $\C=\{e_1, e_2, \ldots,e_{n} \}$.

Up to isometry there are two nondegenerate quadratic forms on $\F_q^n$.
If $n$ is even, these two forms are not similar: we say that the quadratic form has sign $+$ if the dimension of any maximal
totally singular subspace is $\frac{n}{2}$; it has sign $-$ if the dimension of such a space is $\frac{n}{2}-1$.
The corresponding isometry groups are denoted by $\Or_n^+(q)$ and $\Or_n^-(q)$.
If  $n$ is odd, the two quadratic forms are similar. Hence, the corresponding isometry groups are isomorphic and 
will be denoted by $\Or^\circ_n(q)$, or simply by $\Or_n(q)$.
In short, we will write $\Or^\epsilon_n(q)$, where $\epsilon=\circ$ if $n$ is odd, 
$\epsilon = +$ or $\epsilon= -$ if $n$ is even.

If $J$ is the Gram matrix of the symmetric bilinear form $\beta$ associated to a nondegenerate quadratic form $Q$
on $\F_q^n$, we have
$$\beta(v,w)=v^\T  J w \equad 2Q(v)=\beta(v,v) \quad \textrm{ for all } v,w \in \F_q^n.$$
In particular, since $q$ is assumed to be odd, the form $Q$ is determined by $\beta$, i.e., by $J$.
When $n$ is even, the isometry group of $J$ is $\Or_n^+(q)$ if either $\det(J)$ is a square in $\F_q^*$ and 
$\frac{n(q-1)}{4}$ is even or 
$\det(J)$ is a nonsquare and $\frac{n(q-1)}{4}$ is odd;
$\Or_n^-(q)$ otherwise (see \cite[Proposition 1.5.42]{Ho}).

The group $\Omega_n^\epsilon(q)$ is the derived subgroup of $\Or_n^\epsilon(q)$  and has index $2$ in 
$\SO_n^\epsilon(q)$, the subgroup of $\Or_n^\epsilon(q)$ consisting of matrices of determinant $1$.
Alternatively, $\Omega_n^\epsilon(q)$ consists of the elements in $\SO_n^\epsilon(q)$ with spinor norm in $(\F_q^\ast)^2$. 
We recall that the spinor norm $\theta: \Or_n^\epsilon(q) \to \frac{\F_q^\ast}{(\F_q^\ast)^2}$ 
is a homomorphism. 
For any nonsingular $v \in \F_q^n$, the reflection $r_v$, of center $\langle v \rangle$, acts as 
$w\mapsto  w - Q(v)^{-1} \beta(w,v) v$ for all $w \in V$. Moreover, $\theta(r_v) = Q(v)(\F_q^\ast)^2$ 
(see \cite[pages 145, 163 and 164]{T}.

Given an eigenvalue $\lambda$ of $g \in \GL_n(\F)$,
write $V_\lambda(g)$ for the corresponding eigenspace. 
The characteristic polynomial of $g$ will be denoted by $\chi_g(t)$.
Let $\omega\in \F$ be a primitive cube root of $1$.

\begin{lemma}\label{complementL}
Let $H$ be a subgroup of $\GL_n(\F)$ and $U$ be a proper $H$-invariant subspace. 
Suppose that $g\in H$ has the eigenvalue $\lambda\in \F$.
If the restriction $g_{|U}$ does not have the eigenvalue $\lambda$, then there exists a
$H^\T$-invariant subspace $\overline{U}$, with $\dim(\overline{U}) = n-\dim (U)$, such that
$V_\lambda(g^\T) \leq \overline{U}$.
\end{lemma}

\begin{proof}
There exists a nonsingular matrix $P$ such that:
$$P^{-1} H P = \left\{\begin{pmatrix} A_h & B_h \\  0  & C_h \end{pmatrix}
\mid h\in H\right\},\quad
P^\T  H^\T P^{-\T} = \left\{\begin{pmatrix} A_h^\T &  0 \\ B_h^\T  & C_h^\T \end{pmatrix}
\mid h\in H \right\}.$$
Set $A=A_g$, $B=B_g$, $C=C_g$ and $k=\dim(U)$. 
Under our assumption, $A\in \GL_k(\F)$ does not have the eigenvalue $\lambda$. Hence, the same is
true for $A^\T$. So, imposing
$$\begin{pmatrix} A^\T & 0 \\ B^\T & C^\T \end{pmatrix}
\begin{pmatrix} w \\ \overline{w}\end{pmatrix}=
\begin{pmatrix} A^\T w \\ B^\T w + C^\T \overline{w}\end{pmatrix}=
\begin{pmatrix}
\lambda w \\ \lambda\overline{w}
\end{pmatrix},\quad  w \in  \F^k, \; \overline{w} \in \F^{n-k},$$
we get $w = 0$ and
$$V_\lambda(P^\T g^\T P^{-\T}) =\left\{ \begin{pmatrix} 0 \\ \overline{w}\end{pmatrix}\mid C^\T \overline{w}= \lambda \overline{w}\right\}\leq \overline{E}
=\langle e_i\mid k+1\leq i\leq n \rangle.$$
Set $\overline{U} = P^{-\T} \overline{E}$.
Since $\overline{E}$ is invariant under $P^\T H^\T P^{-\T}$,  we get that $\overline{U}$ is 
$H^\T$-invariant.
From $V_\lambda(g^\T ) = P^{-\T} V_\lambda(P^\T g^\T P^{-\T} )$ 
it follows $V_\lambda(g^\T ) \leq \overline{U}$.
\end{proof}

\begin{corollary}\label{complement}
Let $H$ be a subgroup of $\GL_n(\F)$ and $U$ be a proper $H$-invariant subspace. 
Suppose that there exists $J\in \GL_n(\F)$ such that $h^\T J h = J$ for all $h \in H$. 
If  $g\in H$ has the eigenvalue $\lambda\in \F$, then
$$J^{-1}V_\lambda(g^\T)= V_{\lambda^{-1}}(g).$$
Also, if $g_{|U}$ does not have the eigenvalue $\lambda$, then there exists an $H$-invariant subspace $W$, with $\dim(W) = n-\dim (U)$, such that
$V_{\lambda^{-1}}(g) \leq W$.

In particular, for $\lambda=\lambda^{-1}$ (i.e., $\lambda=\pm 1$),  we may assume that  $g_{|U}$ has the eigenvalue $\lambda$.
\end{corollary}

\begin{proof}
From $g^\T J g =J$ we get $g(J^{-1} \overline{s}) =J^{-1} g^{-\T} \overline{s}=\lambda^{-1} (J^{-1}\overline{s})$ for 
all $\overline{s}\in V_{\lambda}(g^\T)$. It follows that $J^{-1} V_{\lambda}(g^\T)\leq V_{\lambda^{-1}}(g)$.
On the other hand, take $v \in V_{\lambda^{-1}}(g)$. Then, $g^\T J v = J g^{-1} v =\lambda J v$
gives $Jv \in V_{\lambda}(g^\T)$, whence $V_{\lambda^{-1}}(g)\leq J^{-1} V_{\lambda}(g^\T)$.

If $g_{|U}$ does not have the eigenvalue $\lambda$, we apply Lemma \ref{complementL}: so, there exists an $H^\T$-invariant subspace $\overline{U}$, with $\dim(\overline{U})=n-\dim(U)$, such that $V_\lambda(g^\T) \leq \overline{U}$.
Set $W=J^{-1}\overline{U}$. For any $h \in H$ we have $h W = h (J^{-1}\overline{U})=J^{-1}h^{-\T}\overline{U}
=J^{-1} \overline{U}=W$. Hence, $W$ is $H$-invariant and  $\dim(W)=\dim(\overline{U})=n-\dim(U)$.
Finally, $V_{\lambda^{-1}}(g) =   J^{-1}V_\lambda(g^\T)\leq J^{-1}\overline{U}=W$.
\end{proof}

To prove our Theorem \ref{main}, we define two elements $x,y$ of respective orders $2$ and $3$, where
$y\in \Omega_n^\epsilon(q)$ and $x$ depends on some parameter $a\in \F_q^*$.  Our aim is to find suitable conditions on $a$ such that 
$x\in  \Omega_n^\epsilon(q)$ and the subgroup $H=\langle x,y\rangle$ is not contained in any maximal subgroup $M$ of 
$\Omega_n^\epsilon(q)$.

The maximal subgroups of classical groups,  described in \cite{Ho,KL},  belong to eight classes $\mathcal{C}_1,\mathcal{C}_2,\ldots,
\mathcal{C}_8$, and a further class $\mathcal{S}$.
Note that, for orthogonal groups, the class $\mathcal{C}_8$ is always empty.
Those which concern our results can be roughly described as follows (see \cite[Table 1.2.A]{KL}):
\begin{itemize}
\item groups which are reducible over $\F$ (classes $\mathcal{C}_1$ and $\mathcal{C}_3$);
\item \emph{imprimitive} groups, i.e. stabilizers of decompositions $\F_q^n=\oplus_{i=1}^t W_i$,
where $\dim(W_i)=\frac{n}{t}$ (class $\mathcal{C}_2$). When $t=n$, they are also called \emph{monomial};
\item stabilizers of subfields of $\F_q$ of prime index (class $\mathcal{C}_5$). They are conjugate to subgroups of $\GL_n(q_0)$
where  $q=q_0^r$ with $r$ prime.
\end{itemize}

To understand these groups, it is also necessary to know the representations of classical groups
in higher dimension, where they may fix nondegenerate forms. 
In  particular, we will need (for instance, in Lemma \ref{9-imtens}), the representation  $\psi: \GL_2(q)\to \GL_3(q)$
arising from the action of $\GL_2(q)$  on the space of homogeneous polynomials of degree $2$
in two variables over $\F_{q}$, namely
\begin{equation}\label{Om3}
\psi\left(\begin{pmatrix}
b_1 & b_2 \\
b_3 & b_4
\end{pmatrix}\right) =
\begin{pmatrix}
b_1^2   & b_1b_2        &  b_2^2\\
2b_1b_3 & b_1b_4+b_2b_3 & 2b_2b_4\\
b_3^2   & b_3b_4        & b_4^2
\end{pmatrix}.
\end{equation}
Note that $\im(\psi)$ preserves the symmetric form 
$\begin{pmatrix}
0&0&1\\
0&-1/2 &0\\
1&0&0
\end{pmatrix}$
whenever $b_1b_4 -b_2 b_3=\pm 1$.

Finally, we recall some well known facts  (e.g., see \cite[page 185]{KL}).
Let $\Sym(\ell)$ be the subgroup of $\GL_\ell(\F)$  consisting of the permutation matrices.
Clearly, $\Sym(\ell)$ preserves  the bilinear form defined by $\I_\ell$. Moreover, it
fixes the vector $u=\sum\limits_{i=1}^\ell e_i$ and the subspace $u^\perp$.

If $p\nmid \ell$, then $u$ is not isotropic, whence $\F^\ell = u^\perp  \operp \left\langle u\right\rangle$.
The restriction of $\Sym(\ell)$ to the subspace $u^\perp$ provides
a representation of $\Sym(\ell)$ of degree $\ell -1$.
The Jordan canonical form of any $\sigma\in \Sym(\ell)$ is obtained from the Jordan form of $\sigma_{|u^\perp}$, adding
one block $\begin{pmatrix} 1\end{pmatrix}$. 

If $p\mid \ell$, then $u\in u^\perp$.
Set $\overline{W}=\langle e_1-e_{i+1}\mid  1\leq i \leq \ell-2\rangle$.
With respect to the decomposition $u^\perp=\overline{W}\oplus \langle u\rangle$,
every $\sigma\in \Sym(\ell)$ has matrix
$$\begin{pmatrix}
\sigma_{|\overline W} & 0 \\
v_\sigma^\T & 1
\end{pmatrix},\quad \sigma_{|\overline W}\in \GL_{\ell-2}(p),\quad  v_\sigma\in \F_p^{\ell-2}.$$
The representation $\sigma\mapsto  \sigma_{|\overline W}$  has degree $\ell-2$.
For any $\sigma$ of order not divisible by $p$, its Jordan form is obtained from
that of $\sigma_{|\overline{W}}$, adding one block $\I_2$.

\section{The case $n\in \{9,11, 13, 17\}$}\label{11-17}

In this section we take $J=\diag\left(\I_{n-3}, \begin{pmatrix}
0&0&1\\
0&1&0\\
1&0&0
\end{pmatrix}\right)$
of determinant $-1$.
For any  $a\in \F_q^*$, we define four matrices $x_1,x_2,y_1,y_2 \in \SL_n(q)$ 
with $x_i^2=y_i^3=\I_n$ as follows.
\begin{itemize}
  \item[($x_1$)]  $x_1$ acts on $\C=\{e_1,\ldots,e_n\}$ as 
\begin{itemize}
\item[$\bullet$] the identity if $n=9$;
\item[$\bullet$] the permutation $(e_1,e_3)(e_2,e_4)$ if $n=11$;
\item[$\bullet$] the permutation $(e_1,e_2)(e_4,e_5)$ if $n=13$;
\item[$\bullet$] the permutation $(e_1,e_3)(e_2,e_4)(e_5,e_6)(e_8,e_9)$ if $n=17$.
\end{itemize}
\item[($x_2$)] $x_2=\diag(\I_{n-9},\bar x)$, where $\bar x=\bar x(a)$ is as in Figure \ref{bar}. 
\item[($y_1$)]  $y_1$ acts on $\C$ as
\begin{itemize}
 \item[$\bullet$] the identity if $n\in \{9,11\}$;
 \item[$\bullet$] the permutation $(e_2,e_3,e_4)$ if $n=13$;
 \item[$\bullet$] the permutation $(e_3,e_4,e_5)(e_6,e_7,e_8)$ if $n=17$.
 \end{itemize}
\item[($y_2)$] $y_2=\diag(\I_{n-9},\bar y)$, where $\bar y$ is as in Figure \ref{bar}. 
\end{itemize}

\begin{figure}[ht]
$$\bar x=\begin{pmatrix}
 1 & 0 & 0 & 0 & 0 & 0 & 0 & 0 & 0 \\
 0 & 1 & 0 & 0 & 0 & 0 & 0 & 0 & 0 \\
 0 & 0 & 0 & 1 & 0 & 0 & 0 & 0 & 0 \\
 0 & 0 & 1 & 0 & 0 & 0 & 0 & 0 & 0 \\
 0 & 0 & 0 & 0 & 0 & 1 & 0 & 0 & 0 \\
 0 & 0 & 0 & 0 & 1 & 0 & 0 & 0 & 0 \\
 0 & 0 & 0 & 0 & 0 & 0 & 0 & 0 & \frac{2}{a}\\
 0 & 0 & 0 & 0 & 0 & 0 & 0 &-1 & 0 \\
 0 & 0 & 0 & 0 & 0 & 0 & \frac{a}{2} & 0 & 0
\end{pmatrix},\;\; \bar y=\begin{pmatrix}
  0&   0&   1&   0&   0&   0&   0&   0&   0\\
  1&   0&   0&   0&   0&   0&   0&   0&   0\\
  0&   1&   0&   0&   0&   0&   0&   0&   0\\
  0&   0&   0&   0&   0&   0&   0&   1&   0\\
  0&   0&   0&   1&   0&   0&   0&   0&   0\\
  0&   0&   0&   0&   0&   -1&   1&   0&   0\\
  0&   0&   0&   0&   0&   -2&   1&   0&   -2\\
  0&   0&   0&   0&   1&   0&   0&   0&   0\\
  0&   0&   0&   0&   0&   0&   -\frac{1}{2}&   0&   0
\end{pmatrix}.$$
\caption{Generators of $\Omega_9(q)$.}\label{bar}
\end{figure}

We can see $x_2$ as the product of an even number of transpositions and $\diag(\I_{n-3}, x_3)$ with
$x_3 =\begin{pmatrix} 0&  0 & \frac{2}{a} \\
0 & -1 & 0 \\ \frac{a}{2}  & 0 &  0
\end{pmatrix}$. 
Identifying $\Sym(n-3)$ with the group of permutation matrices fixing $\{e_j: 1\leq j\leq n-3\}$
and acting as the identity on $\langle e_{n-2},e_{n-1},e_n\rangle$, the first factor of $x_2$
viewed in $\Sym(n-3)\times \GL_3(q)$ is in $\Alt(n - 3) \leq \Omega_n(q)$.
In particular, it is an involution and the same applies to $x_3$.
Similarly, also $x_1$ is the product of an even number of transpositions, so is in  $\Alt(n - 3) \leq \Omega_n(q)$.
Moreover, $x_3\in \Omega_3(q)$
if and only if $-a\in (\F_q^\ast)^2$.  Indeed, $x_3$ is the product of the reflections with centers 
$\langle a e_{n-2}-2e_n\rangle$ and $\langle e_{n-1}\rangle$, whose spinor norms are, respectively, 
$-2a(\F_q^*)^2$ and $\frac{1}{2}(\F_q^*)^2$. 

Clearly, $y_1$ and $y_2$ have determinant $1$. 
Moreover, $y_1 \in \Alt(n - 9)\leq \Omega_n(q)$ and $y_2^\T J y_2=J$. Since 
$x_1x_2=x_2x_1$ and $y_1 y_2 = y_2y_1$, we conclude that $x:= x_1x_2$ and 
$y:= y_1y_2$ have respective orders $2$ and $3$ and
$$H := \langle x,y\rangle \leq \Omega_n(q)\quad \textrm{when } -a \in (\F_q^*)^2.$$
We will also assume that $a\in \F_q^*$ is such that  $\F_p[a]=\F_q$.

By direct computation, we see that the characteristic polynomial of $xy$ is
$$\chi_{xy}(t)= (t + a)(t + a^{-1})(t^{n-2}-1)=t^n + (a + a^{-1})t^{n-1}+t^{n-2}
-t^2-(a+a^{-1})t-1.$$
In particular, $\tr(xy) = -(a + a^{-1})$. Moreover, the minimal polynomial of $xy$ is
$$\textrm{min}_{xy} (t)=
\left\{\begin{array}{cl} 
(t + 1)(t^{n-2}-1)  & \textrm{if } a=1,\\
(t + a)(t + a^{-1})(t^{n-2}-1) & \textrm{otherwise}.
\end{array}       \right.$$
If $a\neq 1$, the minimal polynomial of $xy$ coincides with its characteristic polynomial.
Hence, consideration of the canonical rational form of $xy$ when $a\neq 1$ and direct computation when $a=1$
tell us that $(xy)^{n-2}\neq \I_n$ has a fixed point space of dimension $n-2$, namely it is a \emph{bireflection}.

\begin{lemma}\label{transitivity}
For $1\leq j,k\leq n-3$, there exists $h\in H$ such that $h e_j=e_k$.
\end{lemma} 

\begin{proof}
Clearly, it is enough to show that, for $k\leq n-3$, there exists $h\in H$
such that $he_1=e_k$.  Noting  that $ye_1=e_2$, $ye_2=e_3$, $xe_3=e_4$, $ye_4=e_5$ for $n=9$, 
$xe_1=e_3$, $ye_3=e_4$, $ye_4=e_5$, $xe_4=e_2$ for $n\in\{11,17\}$, and
$xe_1=e_2$, $ye_2=e_3$, $ye_3=e_4$  $xe_4=e_5$ for $n=13$, our claim is true for $k\leq 5$.

Now, let $5 \leq \ell \leq n - 3$ be the largest integer for which, for all $1\leq i\leq \ell$, there exists $h_i \in H$ 
such that $h_i e_1= e_i$.
If $\ell < n - 3$, there exists $h \in \{x, y\}$ such that $he_\ell = e_{\ell+1}$, a contradiction. 
\end{proof}

\begin{lemma}\label{11-irr}
Assume 
$a^2-a-1\neq 0$ if $n=9$, 
$(a-1)(a^3+2a^2+a+1)\neq 0$ if $n=11$, and
$a^4 + a^2 - a + 1\neq 0$ if $n=17$.
Then, the group $H$ is absolutely irreducible. 
\end{lemma}

\begin{proof} 
Assume, for a contradiction, that $U$ is a proper $H$-invariant subspace.
Define  
$$g_9=[x,y],\quad g_{11}=(xy^2)^3xy,\quad g_{13}= (xy^2)^2xy, \quad
g_{17}= (xy^2)^6xy.$$
Under our hypotheses on $a$, for $n=9$  we have $V_1(g_9)= \langle e_1\rangle$. By Corollary \ref{complement} we may assume $e_1\in U$
and hence  $e_1,\ldots,e_{n-3}\in U$ by Lemma \ref{transitivity}.
Similarly, for $n=11$ we have $V_1(g_{11})= \langle e_3\rangle$,
for $n=13$ we have $V_1(g_{13})= \langle e_2\rangle$, and for $n=17$ we have $V_1(g_{17})= \langle e_6\rangle$.
In all these cases, as above, we may assume $e_1,\ldots,e_{n-3}\in U$.
Noting that $y e_{n-3}+e_{n-3} = -2 e_{n-2}$, $y^2 e_{n-5} = e_{n-1}$ and $y^2 e_{n-2} = -\frac{1}{2} e_n$,
we get the contradiction $U=V$.
\end{proof}

For the following result we need the traces of $[x,y]^j$, $j=1,2$:
$$\tr([x,y])= 1+ a^2+ a^{-2}+\varsigma_n \equad \tr([x,y]^2)= (1+a^2+a^{-2})^2 -4a-\kappa_n,$$
where 
$$\varsigma_n=\left\{
\begin{array}{ll}
 1 & \textrm{if } n=9,\\
 0 & \textrm{otherwise}
\end{array}\right.
\equad 
\kappa_n=\left\{
\begin{array}{ll}
 3 & \textrm{if } n=9,\\
 2 & \textrm{if } n=11,\\
 4 &  \textrm{if } n=13,17.
\end{array}\right. $$

\begin{lemma}\label{11-C5}
The group $H$ is not contained in any maximal subgroup $M$ in class $\mathcal{C}_5$  of $\Omega_{n}(q)$.
\end{lemma}

\begin{proof}
Suppose the contrary. By \cite[Tables 8.58 and 8.74]{Ho} and \cite[Proposition 4.5.8]{KL} 
we have either $M\cong \Omega_n(q_0)$ where $q=q_0^r$ and $r$ is an odd prime, or
$M\cong \SO_n(q_0)$ where $q=q_0^2$.
Thus, there exists $g\in \GL_{n}(\F)$ such that
$x^g=x_0$, $y^g=y_0$, with $x_0,y_0\in \GL_{n}(q_0)$.
From  $\tr\left([x,y]^j\right)=\tr\left([x^g, y^g]^j \right)= \tr\left( [x_0,y_0]^j\right)$, $j=1,2$,
it follows that $4a+\kappa_n = (\tr([x,y])-\varsigma_n)^2-\tr([x,y]^2)\in \F_{q_0}$,
whence $a  \in \F_{q_0}$. So,  $\F_q=\F_p[a]\leq \F_{q_0}$ implies $q_0=q$.
\end{proof}

\begin{lemma}\label{mono}
Assume $a^2-a-1\neq 0$  for $n=9$.
If $H$ is absolutely irreducible, then $H$ is not contained in any monomial subgroup of $\Omega_{n}(q)$.
\end{lemma}

\begin{proof}
For the sake of contradiction, suppose that $H$ is contained in a monomial subgroup
$M\in \mathcal{C}_2$ of $\Omega_n(q)$. In this case, we may assume $q=p$ and  $H$ acts  monomially with respect to an
orthonormal basis $\mathcal{B}=\{v_1,v_2,\ldots,v_n\}$, see \cite[Proposition 4.2.15]{KL}.
Moreover, by \cite[Tables 8.58 and 8.74]{Ho} and \cite[Proposition 4.5.8]{KL}, the order of 
$M$  divides $2^{n-1}|\Sym(n)|$. In particular, any prime divisor $\varrho$ of $|H|$ should satisfy $\varrho \leq n$.
If we can show that $e_1\in \mathcal{B}$, we easily get a contradiction. Indeed, from $e_1 \in \mathcal{B}$
it follows $e_i \in \mathcal{B}$ for all $1\leq i \leq n-3$ (see Lemma \ref{transitivity}).
Hence, we may assume $v_i=e_i$ for  $1\leq i\leq n-3$. In particular, $e_{n-3}\in \mathcal{B}$.
As $ye_{n-3}= -2 e_{n-2}-e_{n-3}\not \in \langle e_i \mid 1\leq i \leq n-3\rangle$,
$ye_{n-3}$ should be orthogonal to $v_{n-3}$ obtaining the contradiction 
$v_{n-3}^\T  J y e_{n-3} =e_{n-3}^\T  J y e_{n-3}=  -1\neq 0$.

So, we now show that $e_1\in \mathcal{B}$. To this purpose note that, if $\tr (h)\neq 0$,
then $h$ must fix at least one $\langle v_j\rangle$. 
Moreover, given $h\in H$ of order $k$,
$h\langle v_j\rangle=\langle v_j\rangle$  implies $h v_j=\lambda v_j$, with $\lambda=\pm 1$.
So, consider the permutation $\zeta$ induced by $h$ on the $\langle v_i\rangle$'s.
If $\zeta^b$ acts as the identity on $\{\langle v_1\rangle,\langle v_2\rangle, \ldots, \langle v_n\rangle\}$ for some $b\geq 1$,
then $h^b v_i = \pm v_i$ for every $i$. It follows that $\zeta$ has order $k$ or 
$\frac{k}{2}$.
In particular, if $h$ has odd order, it permutes $\mathcal{B}$ and its cycle structure is determined
by its rational canonical form.
Also, if $h\in H$ does not have the eigenvalue $-1$, from $h\langle v_j\rangle=\langle v_j\rangle$
we get $hv_j=v_j$.  Clearly, this applies to $h=y$.
Since $y$ has order $3$, setting $r=0$ if $n=9$, $r=1$ if $n=13$, and $r=2$ if $n \in \{11,17\}$, 
$y$ fixes $v_j$ for $1\leq j\leq r$ and
permutes the remaining vectors $v_j$ in $\frac{n-r}{3}$ orbits of length $3$.

\noindent \textbf{Case $n=11,13,17$.} 
Call $s$ the number of vectors $u_j=e_j+ y e_j+y^2 e_j$, with $ye_j\neq e_j$, fixed by $y$.
Then, any $v_1 \in V_1(y)$ can be written as
$$v_1=\sum\limits_{i=1}^r \alpha_i e_i+ \sum\limits_{j=1}^s \beta_j u_j.$$
Substituting $e_i$ by $\lambda_i e_i$ and $u_j$ by $\mu_j u_j$ if necessary, 
we may assume that all
the coefficients $\alpha_i,\beta_j$ are in $\{0,1\}$.
Since $y$ fixes $v_1$, by the transitivity of $H$ on the subspaces generated by 
the vectors of $\mathcal{B}$, due to its irreducibility, we may also assume
$v_3=x v_1$, $v_4=yv_3$, $v_5 = y v_4$ and $v_6= x v_5$.
Imposing $v_1^\T Jv_3=v_j^\T J v_6=0$ for all  $j\in \{1,4,5\}$,
we get $v_1\in \{e_1,\ldots,e_r\}$, unless $n \in \{11,17\}$, $q=3$ and $a=-1$.
In these exceptional cases, by direct computation, the order of  $(xy)^2 x y^2$ is divisible by a prime
$\varrho\geq 41$, a contradiction as $\varrho$ does not divide $|\Sym(n)|$, $n\leq 17$ (see the beginning
of the proof).

\noindent \textbf{Case $n=9$.}
Take $h=[x,y]$ and suppose $a^2-a-1\neq 0$. 
Then $V_1([x,y])=\langle e_1\rangle$.
We have $$\tr (h)=\frac{(a^2+1)^2}{a^2} \equad \chi_h(-1)= \frac{-8(a^2+a+1)^2}{a^2}.$$
It follows $e_1\in \mathcal{B}$ unless, possibly, when $a^2+1=0$ or $a^2+a+1=0$.
As previously remarked, the order of any element of $M$, hence a fortiori of $H$,
if odd  must belong to the set $\{ 1, 3,  5, 7,  9, 15 \}$, if prime must belong to  $\{2,3,5,7\}$.
Assume $a^2+1=0$. If $p\neq 5$, we may take $h=[x,y]^2$, as $V_1(h)=\langle e_1\rangle$ and 
$\tr(h)=-4a - 2\neq 0$. If $q=5$, then $a=2$ and $[x,y]$ has order $156=2^2\cdot 3 \cdot 13$, a contradiction. 
So, assume $a^2+a+1=0$. If $p\neq 3$, the permutation induced by $xy$ on the $\langle v_i\rangle$'s
has order divisible by $21$, a contradiction. 
If $q=3$, then $a=1$ and $[x,y]^3y$ has order $41$, a contradiction.
\end{proof}

\begin{lemma}\label{9-imtens}
Assume $n=9$. If the group $H$ is absolutely irreducible, 
then it is neither contained in a maximal subgroup in class $\mathcal{C}_2$ of $\Omega_9(q)$
nor contained in any maximal subgroup in class $\mathcal{C}_7$.
\end{lemma}

\begin{proof}
For the sake of contradiction, suppose that $H$ is imprimitive. By Lemma  \ref{mono}, we may
assume  $H\leq M\cong  \Omega_3(q)^3. 2^4. \Sym(3)$, where $M$ permutes a decomposition 
$\F_q^9 = W_1 \oplus W_2 \oplus W_3$, with $\dim(W_i)=3$.
Set $h=(xy)^7$ and $N=\Omega_3(q)^3$. From $\dim(V_1(h))=7$, we get
$V_1(h)\cap W_i\neq \{0\}$, whence $hW_i=W_i$, for each $i=1,2,3$.
It follows that 
$(xy)^7 \in N$. Since $7$ is coprime to the index of $N$ in $M$, we get $xy \in N$.
Since $y$ acts as a $3$-cycle on $\{W_1,W_2,W_3\}$, it follows that the elements $(xy)^iy$, $1\leq i\leq 7$, have trace equal to zero.
Thus, $0=\tr(xy^2)=-(a+a^{-1})$ gives the condition $a^2+1=0$.
In this case, $\tr((xy)^3y)=1$, a contradiction.

Now, suppose that $H$ is contained in a maximal subgroup $M$ in class $\mathcal{C}_7$ of $\Omega_n(q)$.
By \cite[Table 8.58]{Ho}, $M\cong \Omega_3(q)^2.[4]$. 
Then, $h=(xy)^7$ belongs to $\Omega_3(q)^2$.
Suppose first that $xy$ is semisimple. Up to conjugation, $h=\diag(\beta_1,1,\beta_1^{-1})\otimes \diag(\beta_2, 1, \beta_2^{-1})$,
for some $\beta_1,\beta_2 \in \F_q^*$.
In order that it has the eigenvalue $1$ with multiplicity (at least) $7$, we need $\beta_1=\beta_2=1$, which gives
$h=\I_9$, a contradiction.
Finally, assume that $xy$ has order divisible by $p$.
Up to conjugation and because of \eqref{Om3}, we have

$$h=\begin{pmatrix}
1 & 1 & 1 \\
0 & 1 & 2 \\
0 & 0 & 1 
    \end{pmatrix}
\otimes
\begin{pmatrix}
\beta & 0 & 0 \\
0 & 1 & 0 \\
0 & 0 & \beta^{-1}
\end{pmatrix},
\quad \beta\in \F_q^*.$$
Hence, $\chi_{h}(t)=(t-1)^3(t-\beta)^3(t-\beta^{-1})^3$.
Since $h$ is a bireflection (i.e., $\dim(V_1(h))=7$), we must have $\beta=1$, in which case 
$\dim(V_1(h))=3$, a contradiction.
\end{proof}

\begin{lemma}\label{deleted}
If $H$ is absolutely irreducible, then  the 
$H$-module $V=\F^{n}$ is not the deleted permutation module of degree $\ell=n+1, n+2$.
\end{lemma}

\begin{proof}
Assume the contrary. From  what has been seen at the end of Section \ref{prel},
up to conjugation, we may assume $H\leq \Sym(\ell)\leq \GL_\ell(p)$, with $\ell = n +1, n + 2$. 

\noindent \textbf{Case $\ell=n+1$.}
Fix $h \in H$ such that $\dim(V_1(h))=1$ and call $\zeta$ its preimage in $\Sym(\ell)\leq \GL_\ell(p)$. 
Then, $\zeta$ has at most two orbits. 
It follows that  $\tr(\zeta)=0$ if $\zeta$ is an $\ell$-cycle or the product of two cycles 
of length at least two.
Otherwise, $\tr(\zeta)=1$ and $\zeta$ is a cycle of length $\ell-1$.
Note that $\zeta$ and $h$ have the same order.

We may take $h=xy$, as  $\dim(V_1(xy))=1$. Hence, $\tr(\zeta)-1=\tr(xy)=-(a+a^{-1})$ gives the following two cases:
if $\tr(\zeta)=0$, then $a^2-a+1=0$;
if $\tr(\zeta)=1$, then  $a^2+1=0$.
In the second case, the characteristic polynomial $\chi_{xy}(t)$ is divisible by $t^2+1$, and then $xy$ has order divisible by $4$.
However, $\zeta$ has odd order $n$, being an $n$-cycle, a contradiction.

So, assume $a^2-a+1=0$. In this case, $t^2+t+1$ divides $\chi_{xy}(t)$ and hence the order of $\zeta$ is divisible by $3$.
Furthermore, $(xy)^{n-2}$ has order $p$ when $n\in \{11,17\}$.
For $n=11$, we get that the order of $\zeta$ is $6$, $9$ or $12$, in contrast with $(xy)^9$ of odd order $p$.
For $n=17$, the order of $\zeta$ is $9$, $12$, $15$ or $18$. However, $(xy)^9\neq \I_{17}$ and 
the other values are in contrast with $(xy)^{15}$ of odd order $p$.
For $n\in \{9,13\}$, we apply the previous argument to other elements $h$ such that $\dim(V_1(h))=1$.
For $n=9$, we take $h=[x,y]$ whose trace is equal to $1$, a contradiction.
For $n=13$, we take $h= (xy^2)^2xy$, which has trace equal to $3$.
Since $\tr(h)=\tr(\zeta)-1 \in \{-1,0\}$, we get an absurd unless $p=3$.
However, in this case $a=-1$ and $h^8$ has order $41$, a contradiction as $h^8\in H\leq \Sym(14)$.

\noindent \textbf{Case $\ell=n+2$.}
In this case $q \mid \ell$,  hence we need to consider only the following cases:
(a) $(n,q)=(9,11)$;
(b) $(n,q)=(11,13)$;
(c) $(n,q)=(13,3)$;
(d) $(n,q)=(13,5)$;
(e) $(n,q)=(17,19)$.
Take $g=(xy)^3 (xy^2)^7$ in case (a); $g= xy  (xy^2)^2$ in cases (b), (c) and (e); $g= xy (xy^2)^3$ in case (d).
By direct computation, in all these cases the order of $g$  is divisible by a prime $\varrho\geq n+4$, 
a contradiction as $\varrho$ should divide $|\Sym(n+2)|$.
\end{proof}

\begin{theorem}\label{th-1137} 
Suppose $n \in \{9,11,13,17\}$ and let $a \in \F_q^*$ be such that
$${\rm (i)}\;\; \F_p[a]=\F_q;  \quad {\rm (ii)}\;\; -a\in (\F_q^\ast)^2;
\quad {\rm (iii)}\;\;
\left\{\begin{array}{ll}
a^2-a-1\neq 0 & \textrm{ if } n=9; \\ 
(a-1)(a^3+2a^2+a+1)\neq 0  & \textrm{ if } n=11; \\ 
a^4+a^2-a+1 \neq 0 & \textrm{ if  } n=17.
\end{array}\right.$$
Then, $H=\Omega_{n}(q)$. In particular, $\Omega_{n}(q)$ is $(2,3)$-generated for any $q$ odd.
\end{theorem}

\begin{proof}
By condition (ii), $H$ is a subgroup of $\Omega_{n}(q)$. 
By condition (iii), Lemmas \ref{11-irr}, \ref{mono} and \ref{9-imtens}, the group $H$ is absolutely
irreducible and is neither contained in a maximal subgroup in class $\mathcal{C}_2$ of $\Omega_n(q)$
nor contained in any maximal subgroup in class $\mathcal{C}_7$.
Since it contains the bireflection $(xy)^{n-2}$, we can apply \cite[Theorem 7.1]{GS} which, combined with (i) and Lemma \ref{11-C5},  gives
two possibilities:
(a) $H$ is an alternating or symmetric group of degree $\ell$ and $\F^n$ is the deleted permutation module of dimension $\ell-1$ or $\ell-2$; 
(b) $H=\Omega_{n}(q)$.
Case (a) is excluded by Lemma \ref{deleted}: we conclude that  $H=\Omega_{n}(q)$.

Finally, we have to prove that there exists an element $a$ satisfying all the requirements. If $q=p$, take $a=-1$.
Suppose now $q=p^f$ with $f\geq 2$, and let $\N(q)$ be the number of elements $b\in \F_q^*$ such that $\F_p[b]\neq \F_q$.
By \cite{PTV} we have $\N(q)\leq  p\frac{p^{\lfloor f/2 \rfloor}-1}{p-1}$, hence
it suffices to check when $\frac{p^f-1}{2} - p\frac{p^{\lfloor f/2 \rfloor}-1}{p-1}> 4$.
This condition is fulfilled unless $q=3^2$.
So, assume $q=9$ and take $a\in \F_9^*$ whose minimal polynomial over $\F_3$ is $t^2+1$.
Then, $\F_3[a]=\F_9$ and $-a=(a+1)^2$ is a square.
\end{proof}

\section{Generators for $n\in \{12,15,16\}$ and for $n\geq 18$}\label{gen}

For $n\in \{12,15,16\}$ and for $n\geq 18$, write $n=3m+9+r$, with $m\geq 1$ and $r \in \{0,1,2\}$.
Take the symmetric bilinear form corresponding to the  Gram matrix 
$J=\begin{pmatrix} 
\I_{n-8} &  0 & 0\\ 0 & 0 & \I_4 \\ 0 & \I_4 & 0
 \end{pmatrix}$,
 having $\det(J)=1$.  
For any  $a\in \F_q^*$, we define four matrices $x_1,x_2,y_1,y_2$ of $\GL_n(q)$ as follows.
\begin{itemize}
\item[($x_1$)]  $x_1$ acts on $\C$ as 
the product $\nu_1\nu_2$ of the following two disjoint  permutations:
$$\nu_1=\left\{\begin{array}{cl}
\id & \textrm{if } r=0 \textrm{ and } n \textrm{ is odd},\\
(e_1,e_2) & \textrm{if } r=0 \textrm{ and } n \textrm{ is even},\\
(e_1,e_2) & \textrm{if } r=1 \textrm{ and } n \textrm{ is odd},\\
(e_1,e_2)(e_3,e_6) & \textrm{if } r=1 \textrm{ and } n \textrm{ is even},\\
(e_1,e_3)(e_2,e_4) & \textrm{if } r=2 \textrm{ and } n \textrm{ is odd},\\
(e_1,e_3)(e_2,e_4)(e_7,e_{10}) & \textrm{if } r=2 \textrm{ and } n \textrm{ is even},
\end{array}
\right.$$
and
$$\nu_2= \prod_{j=0}^{m-1} (e_{3j+r+3},e_{3j+r+4})=(e_{r+3},e_{r+4})(e_{r+6},e_{r+7})\cdots (e_{n-9},e_{n-8}).$$
\item[($x_2$)] $x_2=\diag(\I_{n-9},\tilde x)$, where $\tilde x=\tilde x(a)$ is as in Figure \ref{tilde}. 
\item[($y_1$)] $y_1$ 
acts on $\C$ as the permutation
$$\begin{array}{rcl}
\nu_3& =&  \prod\limits_{j=0}^{m-1} \left(e_{3j+r+1},e_{3j+r+2},e_{3j+r+3}\right)\\
 & =& (e_{r+1},e_{r+2},e_{r+3})(e_{r+4},e_{r+5},e_{r+6})\cdots
(e_{n-11}, e_{n-10}, e_{n-9}).
     \end{array}$$
\item[($y_2)$] $y_2=\diag(\I_{n-9},\tilde y)$, where $\tilde y$ is as in Figure \ref{tilde}. 
\end{itemize}

\begin{figure}[ht]
$$\tilde x=\begin{pmatrix}
1 &  0 &   0 &   0 &   0 &   0 &   0 &   0 &   0\\
0 &  0 &   1 &   0 &   0 &   0 &   0 &   0 &   0\\
0 &  1 &   0 &   0 &   0 &   0 &   0 &   0 &   0\\
0 &  0 &   0 &  -1 &   0 &   0 &   0 &   0 &   0\\
0 &  0 &   0 &   a &   1 &   0 &   0 &   0 &   0\\
0 &  0 &   0 &   0 &   0 &   0 &   1 &   0 &   0\\
0 &  0 &   0 &   0 &   0 &   1 &   0 &   0 &   0\\
0 &  0 &   0 &   0 &   0 &   0 &   0 &  -1 &   a\\
0 &  0 &   0 &   0 &   0 &   0 &   0 &   0 &   1
\end{pmatrix},\;\; \tilde y=\begin{pmatrix}
  -1 & 0 & 0 & 0 & 0 & 1 & 0 & 0 & 0\\
   0 & 0 & 0 & 0 & 0 & -\frac{1}{2} & 0 & 0 & 0\\
   0 & 0 & 0 & 0 & 0 & 0 & 0 & 0 & 1\\
   0 & 0 & 0 & 0 & 0 & 0 & 1 & 0 & 0\\
   0 & 0 & 0 & 1 & 0 & 0 & 0 & 0 & 0\\
  -2 &-2 & 0 & 0 & 0 & 1 & 0 & 0 & 0\\
   0 & 0 & 0 & 0 & 1 & 0 & 0 & 0 & 0\\
   0 & 0 & 1 & 0 & 0 & 0 & 0 & 0 &0\\
   0 & 0 & 0 & 0 & 0 & 0 & 0 & 1 & 0
\end{pmatrix}.$$
\caption{Alternative generators of $\Omega_9(q)$.}\label{tilde}
\end{figure}

Let us  identify $\Sym(n-8)$ with the group of permutation matrices fixing $\{e_j: 1\leq j\leq n-8\}$
and acting as the identity on $\langle e_{n-7},e_{n-6},\ldots,e_n\rangle$.
The matrix $x_1$ is the product of $N$ transpositions in $\Sym(n-8)$, where
$N$ is as  follows:
\begin{center}
\begin{tabular}{l|lll}
& $r=0$ & $r=1$ & $r=2$ \\\hline
$n$ even & $N=m+1$ & $N=m+2$ & $N=m+3$ \\
$n$ odd & $N=m$ & $N=m+1$ & $N=m+2$
\end{tabular}
\end{center}
Now, $n$ is odd if and only if $m$ and $r$ have the same parity. It follows that $N$ is always even,
whence $x_1 \in \Alt(n-8)\leq \Omega_n^\epsilon(q)$.
In particular, $x_1$ is an involution and the same is easily verified for $x_2$.
To see that $x_2 \in \Omega_n^\epsilon(q)$, note that $\tilde x=\diag \left(1,h,h^{-\T}\right)$ with $h\in \SL_4(q)$.
Since $\diag \left(1,g,g^{-\T}\right)\in \SO_9(q)$ for each $g\in \GL_4(q)$, we conclude that
$\tilde x$  is in  $\Omega_9(q)$.

Clearly, $y_1$ and $y_2$ have order $3$ and  determinant $1$.
Moreover, $y_1\in \Alt(n-9)\le \Omega_n^\epsilon(q)$, and
$y_2^\T J y_2 =J$.
Since $x_1x_2=x_2x_1$ and $y_1y_2=y_2y_1$, we conclude  that
$x=x_1x_2$ and $y=y_1y_2$ have respective orders $2$ and $3$ and that 
$$H:=\langle x,y\rangle \leq \Omega^\epsilon_n(q).$$
We will also assume that $a\in \F_q^*$ is such that  $\F_p[a]=\F_q$.

When $n\neq 12$, we can decompose  $\F_q^{n}$ into the  direct sum of the following $[x,y]$-invariant subspaces.
Take
$$\mathcal{A}=\left\{\begin{array}{ll}
\langle  e_1, e_3, e_4 \rangle & \textrm{if } n=15,\\
\langle  e_1,  e_5 \rangle \oplus \langle e_2, e_4  \rangle & \textrm{if } n=16,\\
\langle  e_1, e_2, e_4, e_5 \rangle \oplus \langle e_3, e_7, e_8  \rangle & \textrm{if } n=19,\\
\langle  e_1, e_2, e_6, e_8 \rangle \oplus \langle e_3, e_4, e_5,e_9  \rangle & \textrm{if } n=20,\\
\langle  e_1,e_2,e_3,e_4,e_5,e_6,e_8,e_9\rangle \oplus \langle e_7,e_{11},e_{12}\rangle & \textrm{if } n= 23.
\end{array}\right. $$
Otherwise,
$$\mathcal{A} = \left\{\begin{array}{ll}
\langle  e_1, e_2, e_3, e_4, e_6, e_7\rangle & \textrm{if } r=0,\\
\langle  e_1, e_2, e_4, e_5 \rangle \oplus \langle e_3, e_6, e_7, e_8, e_{10}, e_{11} \rangle & \textrm{if } r=1,\\
\langle  e_1,e_2,e_3,e_4,e_5,e_6,e_8,e_9\rangle \oplus \langle e_7,e_{10},e_{11},e_{12},e_{14},e_{15}\rangle & \textrm{if } r=2.
\end{array}\right.$$
Moreover,
$$\begin{array}{rcl}
\mathcal{B}& =& \bigoplus\limits_{j=0}^{m-4-r} \langle e_{5+4r+3j}, e_{9+4r+3j}, e_{10+4r+3j}\rangle, \\
\mathcal{C} &=& \langle e_{n-13}, e_{n-10}, e_{n-9}, e_{n-8}, e_{n-7}, e_{n-6}, e_{n-5}, e_{n-4}, e_{n-3}, e_{n-2}, e_{n-1}, e_n \rangle.
\end{array}$$

\begin{lemma} \label{ana}
Assume $n\neq 12$. Then, $\left([x,y]_{|{\mathcal A}}\right)^{24}=\I$ and $\left([x,y]_{|\mathcal B}\right)^3=\I$.   
\end{lemma}

\begin{proof}
For $n \in \{15,16, 19,20, 23\}$, the element  $[x,y]$ acts on $\mathcal{A}$ as the following permutation
$$\left\{\begin{array}{ll}
(e_3,e_4) & \textrm{if } n =15,\\
(e_1,e_5)(e_2,e_4) & \textrm{if } n =16,\\ 
(e_1,e_5,e_4,e_2) (e_3,e_8,e_7) & \textrm{if } n =19,\\
(e_1, e_6, e_8, e_2)(e_3, e_4,e_9,e_5) & \textrm{if } n=20,\\ 
(e_1,e_6,e_5,e_3,e_4,e_9,e_8,e_2)(e_7,e_{12},e_{11})  & \textrm{if } n =23.
         \end{array}\right. $$
Otherwise, it acts on $\mathcal{A}$ as
$$\left\{\begin{array}{ll}
(e_1,e_4,e_3,e_2,e_7,e_6) &  \textrm{if } n\equiv 0 \pmod 6,\\
(e_1,e_5,e_4,e_2) (e_3,e_8,e_7)(e_6,e_{11},e_{10}) & \textrm{if } n \equiv 1 \pmod 6,\\  
(e_1, e_6, e_8, e_2)(e_3, e_4,e_9,e_5) (e_7,e_{12},e_{11},e_{10},e_{15},e_{14}) & \textrm{if } n\equiv 2\pmod 6,\\
(e_2,e_7,e_6)(e_3,e_4) & \textrm{if } n\equiv 3\pmod 6,\\
(e_1,e_5)(e_2,e_4)(e_3,e_8,e_7,e_6, e_{11},e_{10})  & \textrm{if } n \equiv 4 \pmod 6,\\
(e_1,e_6,e_5,e_3,e_4,e_9,e_8,e_2)(e_7,e_{12},e_{11})(e_{10}, e_{15}, e_{14})  & \textrm{if } n \equiv 5 \pmod 6.  
         \end{array}\right. $$
Finally,  $[x,y]$ acts on each summand of $\mathcal B$ as the $3$-cycle
$\left(e_{5+4r+3j}, e_{10+4r+3j}, e_{9+4r+3j} \right)$.
\end{proof}

By Lemma \ref{ana} and direct computations (in particular, for $n=12$), the element $\tau=[x,y]^{24}$ 
has characteristic polynomial $(t-1)^{n}$. 
More precisely, setting $$\vartheta_0=
\begin{pmatrix}
       1 &       0&     -4a  &      0&  -32a^2 & -36a^2 &    -8a&  -56a^2\\
       0 &       1&     -4a &       0&  -28a^2&  -32a^2 &    -8a&  -64a^2\\
       0 &       0&        1 &       0&      8a&      8a  &      0&     16a\\
       0 &       0&     -8a &       1&  -72a^2&  -64a^2  &  -16a &-128a^2\\
       0&        0 &       0&        0 &       1&        0  &      0 &       0\\
       0&        0&        0&        0 &       0 &       1  &      0&        0\\
       0&        0&        0&        0 &     4a  &    4a   &     1&      8a\\
       0&        0 &       0&        0  &      0&        0   &     0&        1  
\end{pmatrix},$$
we have $\tau=\diag(\I_{n-8}, \vartheta)$, where $\vartheta= \vartheta_0+8(
E_{1,6}+2E_{2,8}+2E_{4,5}- E_{2,5}-2E_{1,8}- 2E_{4,6})$ if $n\in \{12,16,20\}$, $\vartheta=\vartheta_0$ otherwise.
Notice that the minimal polynomial of $\vartheta$ is $(t-1)^3$.
It follows that $\tau$ is an element of order $p$ fixing the $9$-dimensional subspace 
$S_9=\left\langle e_{n-8},e_{n-7},\ldots,e_n \right\rangle$.
Furthermore,  the fixed point space of $\tau_{|S_9}$ has dimension $5$, unless $n\in \{12,16,20\}$ and $a^2= 3$,
in which case it has dimension $7$.

\section{The case  $n\in\{15, 18, 19\}$ or $n\geq 21$}\label{gene}

The subspace $S_9$ is invariant under $K=\langle y, \tau\rangle$: 
our first aim is to find conditions on $a\in \F_q^\ast$ so that  $K_{|S_9}=\Omega_9(q)$.
In the following, we identify $y,\tau$ with their restrictions to $S_9$.

\begin{lemma}\label{irr}
The group $K_{|S_9}$ is absolutely irreducible.
\end{lemma}
 
\begin{proof}
We apply Corollary \ref{complement} to $g=[y, \tau]$ and $\lambda=1$.
So, we may assume that the eigenvector $s=e_{n-8} - e_{n-7}$ is contained in $U$.
Take the matrices $M_1,M_2$, whose columns are the images of $s$ under the following elements:
$$\begin{array}{rl}
M_1: & \I_9,\; y, \; y^2, \;  \tau y^2, \; \tau^2 y^2,\;
y\tau y^2,\; y^2\tau y^2,\; y\tau^2 y^2,\; y^2\tau^2 y^2;\\
M_2: &  \I_9,\; y, \; y^2, \;  \tau y^2, \; \tau^2 y^2,\;
y\tau y^2,\; y\tau^2 y^2,\; (\tau y^2)^2,\; \tau y^2\tau^2 y^2.  
  \end{array}$$
Then, $\det(M_1)= -2^{35} a^{10} (4 a^2+3)$
and $\det(M_2)=-2^{35} a^{10} (28a^2 - 3)$.
Clearly, these two matrices cannot be both singular, whence $\dim(U)=9$, a contradiction.
\end{proof}

\begin{lemma}\label{monS}
The group $K_{|S_9}$ is neither monomial nor contained in any maximal subgroup
$\PSL_2(8)$, $\PSL_2(17)$, $\Alt(10)$, $\Sym(10)$, $\Sym(11)$ in  class $\mathcal{S}$ of $\Omega_9(q)$.
\end{lemma}

\begin{proof}
Recall that $\tau$ is an element of order $p$. Considering the order of the maximal
subgroups $M$ described in the statement and the conditions on $q$ given in \cite[Tables 8.58 and 8.59]{Ho},
we may  reduce to the following cases:
\begin{itemize}
\item[(i)] $M=2^8: \Alt(9)$ and $q\in \{3,5\}$;
\item[(ii)] $M=2^8: \Sym(9)$ and $q=7$;
\item[(iii)] $M=\Alt(10)$ and  $q \in \{3,7\}$;
\item[(iv)] $M=\PSL_2(17)$ and $q=9$;
\item[(v)] $M=\Sym(11)$ and $q=11$;
\item[(vi)] $M=\PSL_2(8)$ and $q=27$.
\end{itemize}
Now, we look for an element of $H$ whose order does not divide $|M|$.
In particular, it suffices to find an element of $H$ whose order is divisible by a prime $\varrho >17$ in case
(iv), $\varrho >11$ otherwise.
Define $g_j=y\tau^j$.
If $q\in \{3,9\}$, then $g_1$ has order divisible by $41$.
If $q=5$, then $g_3$ has order which is divisible by a prime $\varrho\geq 13$.
If $q=7$, take $j=2$  when $a=\pm 2$, and $j=3$ when $a\in \{\pm 1, \pm 3 \}$.
Then, the order of $g_j$ is divisible by a prime $\varrho\geq 43$.
If $q=11$, take $j=2$ if $a= \pm 5$ and $j=1$ otherwise. Then the order of $g_j$ is divisible by a prime $\varrho \geq 19$.
Finally, if $q=27$ then $g_2$ has order divisible by $37$.
In all these cases, we easily obtain a contradiction.
\end{proof}

For the next lemma, we will use the following traces of elements of  $K_{|S_9}$:
\begin{equation}\label{tr1}
\tr\left((y\tau)^2\right) = -2176 a^4 + 128 a^2, \qquad 
\tr\left((y^2\tau)^2\right) = 1920 a^4 + 128 a^2.
\end{equation}

\begin{lemma}\label{imp}
The group $K_{|S_9}$ is neither contained in a maximal subgroup in class $\mathcal{C}_2$ of $\Omega_9(q)$
nor contained in any maximal subgroup in class $\mathcal{C}_7$.
\end{lemma}

\begin{proof}
By Lemma \ref{monS} the group $K_{|S_9}$ is not monomial.
So, suppose that $K_{|S_9}$ preserves a nonsingular decomposition $\F_q^9=W_1\oplus W_2\oplus W_3$
with $\dim(W_i)=3$. Clearly, for each $k\in K_{|S_9}$, its cube fixes each $W_i$ preserving a nonsingular
symmetric form. Thus, its eigenvalues are $\pm 1, \alpha_i, \alpha_i^{-1}$.
It follows that $k^3$ must have the eigenvalue $1$ with multiplicity at least $3$, or the eigenvalue
$-1$ with multiplicity at least $2$.
Assume first $p=3$. We have $\chi_{(y\tau)^3}(t)=(t-1)f(t)$, where
$f(t)=t^8 + t^7 -(a^{12} +a^6 - 1) t^6 - (a^{12} - 1) t^5 - (a^6 - 1) t^4
-(a^{12} - 1) t^3 - (a^{12} + a^6 - 1) t^2 + t + 1$.
Then, $f(1)=-a^{12}\neq 0$ and $f(-1)=1$, a contradiction.
Next, assume $p\neq 3$. From $\tr(\tau)=9\neq 0$ we get that $\tau$ fixes each $W_i$. 
By the irreducibility of $K_{|S_9}$ the element $y$ acts on $\{W_1,W_2,W_3\}$ as the $3$-cycle $(W_1,W_2,W_3)$.
In this case, both $(y\tau)^2$ and  $(y^2\tau)^2$ should have trace $0$, in contrast
with \eqref{tr1} which gives $0=\tr((y^2\tau)^2)-\tr((y\tau)^2)=2^{12}a^4$.

Finally, suppose that $K_{|S_9}$ is contained in a maximal subgroup $M\cong \Omega_3(q)^2.[4]\in \mathcal{C}_7$, hence
actually in $\Omega_3(q)^2$.
Up to conjugation, we may suppose
$\tau=\begin{pmatrix}
1 & 1 & 1 \\
0 & 1 & 2 \\
0 & 0 & 1 
    \end{pmatrix}
\otimes
\begin{pmatrix}
1 & 1 & 1 \\
0 & 1 & 2 \\
0 & 0 & 1 
    \end{pmatrix}$.
The dimensions of the fixed point space of this tensor product and of $\tau$ are, respectively, $3$ and
$5$, a contradiction.
\end{proof}

\begin{lemma}\label{L2}
The group $K_{|S_9}$ is not contained in any maximal subgroup 
$M\cong \PSL_2(q).2$ or $M\cong \PSL_2(q^2).2$ in class $\mathcal{S}$ of $\Omega_9(q)$. 
\end{lemma}

\begin{proof}
Suppose the contrary. 

\noindent \textbf{Case $M\cong \PSL_2(q).2$.} In this case, $M$ arises from the representation
$\Phi: \GL_2(q)\to \GL_9(q)$ obtained from the action of $\GL_2(q)$ on the space $T$
of homogeneous polynomials of degree $8$ in two variables $t_1, t_2$ over $\F_q$. 
Up to conjugation in $\GL_2(q)$, we may assume
$$\tau=\Phi(\I_2+E_{1,2})=\left\{
\begin{array}{rcl}
t_1 & \mapsto & t_1,\\
t_2 & \mapsto & t_1+t_2.
\end{array}\right.$$
Direct computation (with respect to the basis $t_1^8, t_1^7t_2, \ldots, t_2^8$ of $T$)
gives that the fixed point space of this linear transformation is generated by $t_1^8$. 
So, it has dimension $1$, a contradiction as $\tau$ has a fixed point space of dimension $5$.

\noindent \textbf{Case $M\cong \PSL_2(q^2).2$.}  
To understand $M$, start from the representation $\psi:\GL_2(q^2)\to \GL_3(q^2)$ described in \eqref{Om3}. 
Next, consider the subspace $W$ of $\Mat_3(q^2)$ consisting of the matrices $A$ such that $A^\T=(a_{i,j}^q)=A^\sigma$.
Clearly,  $W$ has dimension $9$ over $\F_q$ and we may consider the representation $\Phi: \GL_3(q^2)\to \GL_9(q)$
induced by  $A\mapsto (\psi(g))^\T A (\psi(g))^\sigma$ for all $g\in\GL_3(q^2)$.
The group $M$ arises from this representation. Again, up to conjugation in $\GL_2(q^2)$, we may suppose
$\tau=\Phi(\psi(\I_2+E_{1,2}))=\Phi\left(\begin{pmatrix}
1 & 1 & 1 \\
0 & 1 & 2 \\
0 & 0 & 1
\end{pmatrix} \right)$.
Direct calculation gives that the fixed point space of $\Phi(\psi(\I_2+E_{1,2}))$ on  $W\leq \Mat_3(q^2)$ is generated
by $E_{2,2}, E_{3,3}, E_{2,3}+E_{3,2}$.
Thus, it has dimension $3$, again a contradiction  as $\tau$ has a fixed point space of dimension $5$.
\end{proof}

\begin{proposition}\label{Om9-K9}
Suppose $q$ odd and $n\in\{  15, 18,19 \}$ or $n\geq 21$. 
Let $a \in \F_q^*$ be such that $\F_p[a]=\F_q$. Then, $K_{|S_9}=\Omega_9(q)$.
\end{proposition}

\begin{proof}
By Lemmas \ref{irr} and \ref{imp}, $K_{|S_9}$ is absolutely irreducible,
and is neither contained in a maximal subgroup in class $\mathcal{C}_2$ of $\Omega_9(q)$
nor contained in any maximal subgroup in class $\mathcal{C}_7$.
Furthermore, by Lemmas \ref{monS} and \ref{L2}  either $K_{|S_9}=\Omega_9(q)$ or $K_{|S_9}$ is contained in a maximal subgroup
$M\in \{\Omega_9(q_0), \SO_9(q_0)\}$  in class $\mathcal{C}_5$, where $q=q_0^r$ for some prime $r\geq 2$.
Suppose there exists $g\in \GL_{9}(\F)$ such that
$\tau^g=\tau_0$, $y^g= y_0$, with $\tau_0,y_0\in \GL_{9}(q_0)$.
From  $ -2176 a^4 + 128 a^2 = \tr\left((y\tau)^2\right)= \tr\left((y^g\tau^g)^2 \right)= \tr\left((y_0\tau_0)^2\right)$, 
it follows that $17 a^4- a^2  \in \F_{q_0}$.
Similarly, from $\tr\left((y^2\tau)^2\right) = 1920 a^4 + 128 a^2$, we obtain  
$15 a^4+a^2 \in \F_{q_0}$.
It follows that $32a^4\in \F_{q_0}$ and then $a^2 \in \F_{q_0}$.
Again, from $\tr\left( y^2\tau^2  (y\tau)^2\right)=
-49152 a^6 + 16384 a^5 + 3840 a^4 + 256 a^2 \in \F_{q_0}$, we get 
$a\in\F_{q_0}$. 
So,  $\F_q=\F_p[a]\leq \F_{q_0}$ implies $q_0=q$, a contradiction.
We conclude that  $K_{|S_9}=\Omega_9(q)$.
\end{proof}

Define $E_0=S_0=\{0\}$  and, for $1\leq \ell\leq n$, 
$$E_\ell= \left\langle e_i\mid 1\leq i\leq \ell\right\rangle \equad S_\ell= \left\langle e_i\mid n-\ell+1\leq i\leq 
n\right\rangle.$$

\begin{corollary}\label{cor_gen}
Suppose $q$ odd and $n\in\{  15, 18,19 \}$ or $n\geq 21$. Let $a \in \F_q^*$ be such that $\F_p[a]=\F_q$. Then,
\begin{itemize}
\item[(i)]   $H=\Omega_n (q)$ if $n$ is odd;
\item[(ii)]  $H=\Omega_n^+(q)$ if $q\equiv 1 \pmod 4$ and $n$ is even;
\item[(iii)] $H=\Omega_n^+(q)$ if $q\equiv 3 \pmod 4$ and $n\equiv 0 \pmod 4$;
\item[(iv)] $H=\Omega_n^-(q)$ if $q\equiv 3 \pmod 4$ and $n\equiv 2 \pmod 4$.
\end{itemize}
\end{corollary}

\begin{proof} 
By \cite[Proposition 1.5.42(ii)]{Ho}, when $n$ is even, we have $H\leq \Omega_n^+(q)$ or $H\leq \Omega_n^-(q)$ according as
$\frac{n(q-1)}{4}$ is even or odd, respectively. Let $\ell$ be maximal with respect to
$$K_\ell:=\diag(\I_{n-\ell}, \Omega_{\ell}^\epsilon (q))\leq H,$$
where $\epsilon \in \{\circ,\pm\}$.
Noting that $K'=\diag(\I_{n-9}, \Omega_9(q))$ by the previous proposition,
we have that $\ell$ is at least $9$ and we need to show that $\ell=n$.
For the sake of contradiction, assume $9\leq \ell< n$.

Suppose first that $(r,\ell)\not \in \{(2,n-2), (2,n-1) \}$, and $(r,\ell)\not \in \{(1,n-4),(2,n-8)\}$ when $n$ is even. Then,
\begin{itemize}
\item[{\rm (a)}] if $\ell \equiv 0 \pmod 3$, then $x$ fixes the subspaces $S_{\ell-1}$ and $E_{n-\ell-1}$, and acts as the transposition 
$(e_{n-\ell},e_{n-\ell+1})$ on $\langle e_{n-\ell}, e_{n-\ell+1} \rangle$;
\item[{\rm (b)}] if $\ell \equiv j \pmod 3$, with $j=1,2$, then $y$ fixes the subspaces $S_{\ell-j}$ and $E_{n-\ell-3 +j}$, 
and acts as the cycle 
$\left(e_{n-\ell-2+j},e_{n-\ell-1+j}, e_{n-\ell+j}\right)$ on $\langle e_{n-\ell-2+j},e_{n-\ell-1+j}, e_{n-\ell+j}\rangle$.
\end{itemize}

Setting $g=x$ in case (a), $g=y$ in case (b), we claim that
$K_{\ell +1}:=\left\langle K_\ell, K_{\ell}^g \right\rangle$ equals 
\begin{equation}\label{ell+1}
\diag(\I_{n-\ell-1}, \Omega_{\ell+1}^{\overline{\epsilon}}(q)),\quad \overline{\epsilon}\in \left\{\circ, \pm \right\}.
\end{equation}
Noting  that $g^{-1}S_\ell$ is obtained from $S_\ell$ replacing $e_{n-\ell+1}$ by $e_{n-\ell}$,
one gets $\langle S_\ell,  g^{-1} S_\ell\rangle = S_{\ell +1}$.
Thus, $K_{\ell+1}$ fixes $S_{\ell+1}$,  induces the identity on $E_{n-\ell-1}$
and fixes the restriction of $J$ to $S_{\ell+1}$, of determinant $1$.
If follows that $K_{\ell+1}$ is contained in the group \eqref{ell+1}.
Call $\rho$ the matrix in $\GL_n(q)$ which acts according to
$e_{n-\ell}\mapsto -e_{n-\ell}$, $e_{n-4}\mapsto -2e_n$, $e_n\mapsto -\frac{1}{2} e_{n-4}$ and
fixes the remaining vectors $e_i$.
Since $\rho$ has determinant $1$ and spinor norm $(\F_q^*)^2$, it belongs to $K_\ell^g$ which induces 
$\Omega_{\ell}^\epsilon$ on $g^{-1}S_\ell$.
Now, $\langle \rho, K_\ell\rangle$ is the stabilizer in the group \eqref{ell+1} of the nondegenerate subspace $\langle e_{n-\ell}\rangle$.
So, it is a maximal subgroup of the group \eqref{ell+1}. From $K_{\ell+1}\nleq \langle \rho, K_\ell\rangle$, we get the final contradiction 
$K_{\ell +1}=\diag(\I_{n-\ell-1}, \Omega_{\ell+1}^{\overline{\epsilon}}(q))$.

It remains to exclude the exceptional cases: in each of them, we get the same contradiction.

\noindent\textbf{Case 1.} $r=1$, $\ell=n-4$, $n$ even. Let $R$ be the stabilizer of $e_6$ in $K_{n-4}$.
Then, $\langle R^x, K_{n-4}\rangle = K_{n-3}$, as it fixes the vectors $e_1,e_2,e_3$ and the subspace $E_3^\perp$, inducing $\Omega_{n-3}(q)$.

\noindent\textbf{Case 2.} $r=2$, $\ell=n-8$, $n$ even. Let $R$ be the stabilizer of $e_{10}$ in $K_{n-8}$.
Then, $\langle R^x, K_{n-8}\rangle = K_{n-7}$, as it fixes the vectors $e_1,e_2,\ldots,e_7$ and  the subspace $E_7^\perp$, inducing
$\Omega_{n-7}(q)$.

\noindent\textbf{Case 3.} $r=2$, $\ell=n-2$. Let $R$ be the stabilizer of $e_{3}$ in $K_{n-2}$.
Then, $\langle R^x, K_{n-2}\rangle = K_{n-1}$, as it fixes  $e_1$  and  $E_1^\perp$, inducing $\Omega^{\overline \epsilon}_{n-1}(q)$.

\noindent\textbf{Case 4.} $r=2$, $\ell=n-1$. Similar to the above cases.
\end{proof}

\section{The case $n\in \{12, 16, 20\}$}\label{16-20}

The values $n=12,16,20$ require some small adjustments with respect to the 
general case, described in Section \ref{gene}. So, in the proof of the following results, we only give the necessary
modifications.

\begin{lemma}\label{irr-1620}
Assume  $a^2\neq 2,3$. 
Then, the group $K_{|S_9}$ is absolutely irreducible.
\end{lemma}
 
\begin{proof}
We have $s=e_{n-8} - e_{n-7}$ by the hypothesis $a^2\neq 3$.
Now, $\det(M_1)=-2^{35} a^6 (a^2 - 2) (4 a^4-13 a^2+16) $ and
$\det(M_2)=-2^{35} a^6 (a^2 - 2) (28 a^4 - 83 a^2 - 16)$.
Since $a^2\neq 2$, the matrices $M_1,M_2$ are both singular only if $p=13$ and $a^2=3$, which is excluded by hypothesis.
\end{proof}

\begin{lemma}\label{monS-1620}
The group $K_{|S_9}$ is neither monomial nor contained in any maximal subgroup
$\PSL_2(8)$, $\PSL_2(17)$, $\Alt(10)$, $\Sym(10)$, $\Sym(11)$ in  class $\mathcal{S}$ of $\Omega_9(q)$.
\end{lemma}

\begin{proof}
If $q\in \{3, 5, 11\}$ proceed as in the proof of Lemma \ref{monS}.
If $q=7$,  take $j=1$ if $a=\pm 1$, and $j=3$ if $a=\pm 2$; take $\tilde g=\tau^2 y \tau y$ if $a=\pm 3$.
Then, the order of $g_j$ and the order of $\tilde g$ are divisible by a prime $\varrho\geq 43$.
If $q=9$, then $g_1$ has order divisible by $13$, a prime which does divide $|\PSL_2(17)|$; 
if $q=27$,  then $g_2$ has order divisible by a prime $\varrho \in \{13, 73\}$.
\end{proof}

\begin{lemma}\label{imp-1620}
Assume $a^2 \neq 3$.
The group $K_{|S_9}$ is neither contained in a maximal subgroup in class $\mathcal{C}_2$ of $\Omega_9(q)$
nor contained in any maximal subgroup in class $\mathcal{C}_7$.
\end{lemma}

\begin{proof}
We proceed as in the proof of Lemma \ref{imp}, describing only the necessary modifications to prove the
primitivity of $K_{|S_9}$.
For $p=3$ we have $\chi_{(y\tau)^3}(t)=(t-1)f(t)$, where
$f(t)=t^8 -t^7 - (a^{12} - a^6 +1) t^6 - a^{12} t^5 + (a^6 -1) t^4 - a^{12} t^3 - (a^{12} - a^6 + 1) t^2 -t + 1$.
Also in this case, $f(1)=-a^{12}$ and $f(-1)=1$.
If $p\neq 3$, the product $y\tau$ should have trace $0$, in contrast with $\tr(y\tau) =-16$.
\end{proof}

\begin{lemma}\label{L2-1620}
Assume $a^2 \neq 3$.
The group $K_{|S_9}$ is not contained in any maximal subgroup $M\cong \PSL_2(q).2$ or $M\cong \PSL_2(q^2).2$ in class $\mathcal{S}$ 
of $\Omega_9(q)$. 
\end{lemma}

\begin{proposition}
Assume $q$ odd and $n\in\{12,16,20\}$.
Let $a \in \F_q^*$ be such that $\F_p[a^2]=\F_q$ with $a^2\neq 2,3$.
Then  $K_{|S_9}=\Omega_9(q)$.
\end{proposition}

\begin{proof}
By Lemmas \ref{irr-1620} and \ref{imp-1620}, $K_{|S_9}$ is absolutely irreducible 
and is neither contained in a maximal subgroup in class $\mathcal{C}_2$ of $\Omega_9(q)$
nor contained in any maximal subgroup in class $\mathcal{C}_7$.
Furthermore, by Lemmas \ref{monS-1620} and \ref{L2-1620}  either $K_{|S_9}=\Omega_9(q)$ or $K_{|S_9}$ is contained in a maximal subgroup
$M\in \{\Omega_9(q_0), \SO_9(q_0)\}$  in class $\mathcal{C}_5$, where $q=q_0^r$ for some prime $r\geq 2$.
Suppose there exists $g\in \GL_{9}(\F)$ such that
$\tau^g=\tau_0$, $y^g= y_0$, with $\tau_0,y_0\in \GL_{9}(q_0)$.
From $\tr\left((y\tau)^2\right)=-2176a^4 + 6784a^2 - 224$
and $\tr\left((y^2\tau)^2\right)=1920a^4 - 5504a^2 - 288$, we get that $-17a^4+53a^2$ and $15a^4-43a^2$ belong to $\F_{q_0}$, 
whence $64a^2\in \F_{q_0}$.
We conclude that  $K_{|S_9}=\Omega_9(q)$.
\end{proof}

\begin{corollary}\label{cor_121620}
Assume $q$ odd and $n\in\{12,16,20\}$. Let $a \in \F_q^*$ be such that $\F_p[a^2]=\F_q$ with $a^2\neq 2,3$.
Then   $H=\Omega_{n}^+(q)$.
In particular, $\Omega_n^+(q)$ is $(2,3)$-generated.
\end{corollary}

\begin{proof}
Since $K_{|S_9}=\Omega_9(q)$, we can repeat the argument of Corollary \ref{cor_gen}, proving that $H=\Omega_{n}^+(q)$.
For the second part of the statement, we have to prove that there exists an element $a$ satisfying all the hypotheses. 
If $q=p$, take $a=1$.
Suppose now $q=p^f$ with $f\geq 2$, and let $\N(q)$ be the number of elements $b\in \F_q^*$ such that $\F_p[b^2]\neq \F_q$.
By \cite[Lemma 2.7]{PTV} it suffices to check that the condition $p^f   - 2p \frac{p^{\lfloor f/2 \rfloor}-1}{p-1}> 1$
is always fulfilled (the requirements $a^2\neq 2,3$ can be dropped).
\end{proof}

\section{Conclusions}
 
We can now prove our main result.

\begin{proof}[Proof of Theorem \ref{main}.] The $(2,3)$-generation of $\Omega_n(q)$, $nq$ odd, follows from  Theorem \ref{th-1137} when $n\in \{9,11,13,17\}$,  and Corollary \ref{cor_gen} for the other values of $n$. Due to Corollaries \ref{cor_gen} and \ref{cor_121620}, we also proved the $(2,3)$-generation of the following
even-dimensional orthogonal groups:
$\Omega_{2k}^+(q)$, when $q\equiv 1 \pmod 4$ and $k=6$ or $k\geq 8$;
$\Omega_{4k}^+(q)$, when $q\equiv 3\pmod 4$ and $k\geq 3$;
$\Omega_{4k+2}^-(q)$, when $q\equiv 3\pmod 4$ and $k\geq 4$.
\end{proof}

\end{document}